\documentclass[12pt]{amsart}
\usepackage[margin=1.15in]{geometry}
\usepackage{amscd, amssymb, amsmath, wasysym}
\usepackage{graphicx}
\usepackage{amsfonts}
\usepackage{mathrsfs}    
\usepackage{eucal}     
\usepackage{latexsym}   
\usepackage{verbatim}   
\usepackage[all]{xy}   
\usepackage[dvipsnames]{xcolor}
\usepackage{bookmark}
\usepackage{subcaption}\usepackage{BOONDOX-calo}

 \usepackage{hyperref}
 \hypersetup{
     colorlinks=true,
     linkcolor=NavyBlue,
     filecolor=NavyBlue,
     citecolor = TealBlue,
     urlcolor=magenta,
  }

\theoremstyle{definition}

\theoremstyle{remark}

\theoremstyle{corollary}

\theoremstyle{theorem}

\theoremstyle{corollary}

\newtheorem{theorem}{Theorem}[section]
\newtheorem{lemma}[theorem]{Lemma}
\newtheorem{proposition}[theorem]{Proposition}
\newtheorem{conjecture}[theorem]{Conjecture}

\theoremstyle{corollary}

\theoremstyle{definition}

\theoremstyle{remark}
\newtheorem{remark}[theorem]{Remark}

\numberwithin{equation}{section}

\newcommand{\C}{\mathbf{C}}

\newcommand{\Z}{\mathbf{Z}}

\newcommand{\Q}{\mathbf{Q}}
\newcommand{\sO}{\mathscr{O}}
\def\nP{\mathbf{P}}

\def\Spec{\operatorname{Spec}}

\def\Sing{\operatorname{Sing}}

\def\Supp{\operatorname{Supp}}

\def\vol{\operatorname{vol}}
\def\lct{\operatorname{lct}}

\def\ord{\operatorname{ord}}

\def\Aut{\operatorname{Aut}}
\def\PGL{\operatorname{PGL}}

\newcommand{\suchthat}{\;\ifnum\currentgrouptype=16 \middle\fi|\;}

\input xy
\xyoption{all}

\title[K-polystability of secant varieties of rational curves]{K-polystability of the first secant varieties\\ of rational normal curves}

\begin{document}

\author{In-Kyun Kim}
\address{June E Huh Center for Mathematical Challenges, KIAS, 85 Hoegiro Dongdaemun-gu, Seoul 02455, Republic of Korea}
\email{soulcraw@gmail.com}

\author{Jinhyung Park}
\address{Department of Mathematical Sciences, KAIST, 291 Daehak-ro, Yuseong-gu, Daejeon 34141, Republic of Korea}
\email{parkjh13@kaist.ac.kr}

\author{Joonyeong Won}
\address{Department of Mathematics, Ewha Womans University, 52 Ewhayeodae-gil, Seodaemun-gu, Seoul 03760, Republic of Korea}
\email{leonwon@ewha.ac.kr}

\date{\today}

\thanks{The authors were partially supported by the National Research Foundation (NRF) funded by the Korea government (MSIT) (NRF-2022M3C1C8094326).}


\begin{abstract}
The first secant variety $\Sigma$ of a rational normal curve of degree $d \geq 3$ is known to be a $\Q$-Fano threefold. In this paper, we prove that $\Sigma$ is K-polystable, and hence, $\Sigma$ admits a weak K\"{a}hler--Einstein metric. We also show that there exists a $(-K_{\Sigma})$-polar cylinder in $\Sigma$.
\end{abstract}

\maketitle


\section{Introduction}
The secant varieties of projective varieties have attracted considerable attention in the last three decades partly because they have found important applications to algebraic statistics and algebraic complexity theory. In recent years, there has been a great deal of progress on the syzygies of secant varieties of smooth projective curves (\cite{CKP}, \cite{ENP1}). In order to approach algebraic properties such as syzygies through geometric methods, it is very important to control the singularity of the secant varieties. For this purpose, Ein--Niu--Park \cite[Theorem 1.1]{ENP1} established that the secant varieties of smooth projective curves of large degree have normal Cohen--Macaulay Du Bois singularities. Furthermore, it was also shown that the secant varieties of rational normal curves are $\Q$-Fano varieties. It is an interesting problem to study the singularities of pluri-anticanonical linear systems of secant varieties of rational normal curves from a birational geometry viewpoint. In this paper, we focus on the first secant varieties.

\medskip

Throughout the paper, we work over the field $\C$ of complex numbers. Let
$$
C :=v_d(\nP^1) \subseteq \nP H^0(\nP^1, \sO_{\nP^1}(d)) = \nP^d
$$
be a rational normal curve of degree $d \geq 3$, and
$$
\Sigma:=\Sigma_1(\nP^1, \sO_{\nP^1}(d)) \subseteq  \nP^d
$$
be the \emph{first secant variety} of $C$ in $\nP^d$, which is the closure of the union of secant lines to $C$ in $\nP^d$. Note that $\Sigma$ is precisely the union of bi-secant and tangent lines to $C$ in $\nP^d$. If $d=3$, then $\Sigma = \nP^3$. Assume henceforth that $d \geq 4$. Then
$$
\dim \Sigma = 3~~\text{ and }~~\Sing \Sigma = C.
$$ 
Recall from \cite[Theorem 1.1]{ENP1} that $\Sigma$ is a $\Q$-factorial normal projective variety of Picard number 1 such that $\Sigma$ has log terminal singularities and $-K_{\Sigma}$ is ample. Let $H$ be a hyperplane section of $\Sigma$ in $\nP^d$. We have
$$
-K_{\Sigma} \sim_{\Q} \frac{4}{d-2}H~~\text{ and }~~(-K_{\Sigma})^3 = 32\frac{(d-1)}{(d-2)^2}.
$$
Let $T$ be the \emph{tangent developable surface} of $C$ in $\nP^d$, which is the union of the tangent lines to $C$ in $\nP^d$. We have $\Sing T = C$.
A general hyperplane section of $T$ is a canonically embedded $d$-cuspidal rational curve of degree $2d-2$, which is a degeneration of a general canonical curve of genus $d$. It has long been known that generic Green's conjecture on syzygies of canonical curves \cite[Conjecture 5.1]{Green} can be proven using tangent developable surfaces of rational normal curves (see \cite{EL}). The syzygies of tangent developable surfaces of rational normal curves were completely determined by Aprodu--Farkas--Papadima--Raicu--Weyman \cite{AFPRW}. This gives an alternative proof of generic Green's conjecture, which was first shown by Voisin  \cite{Voisin1, Voisin2}. In \cite{Park}, based on the observation that
$$
T \in |-K_{\Sigma}|,
$$ 
the second author of the present paper provided a simple proof of the aforementioned result of \cite{AFPRW} by presenting a new geometric approach to syzygies of tangent developable surfaces. 

\medskip

First, we study K-polystability of secant varieties of rational normal curves. There has been a great deal of intensive research on the existence of K\"{a}hler--Einstein metrics on Fano varieties, which is a central topic in differential geometry and complex geometry (e.g.,  \cite{CalabiProblem}, \cite{CS}, \cite{Donaldson2}, \cite{Fujita1, Fujita2}, \cite{Kim2018, Kim2021b, Kim2023a, Kim2022, Kim2023, Kim2021a}, \cite{Tian1}). The problem classifying K\"{a}hler--Einstein Fano manifolds, known as the Calabi problem, was completely resolved by Tian \cite{Tian1} for del Pezzo surfaces, and there has been a great progress for smooth Fano threefolds (see \cite{CalabiProblem}). The concept of K-stability was first introduced by Tian \cite{Tian2} and algebraically reformulated by Donaldson \cite{Donaldson1} to give a characterization of the existence of K\"{a}hler--Einstein metric. The famous Yau--Tian--Donaldson conjecture predicts that if $X$ is a Fano variety, then $X$ admits a (weak) \emph{K\"{a}hler--Einstein metric} if and only if $X$ is \emph{K-polystable}. This conjecture was eventually verified in \cite{CDS}, \cite{Tian3} when $X$ is smooth and in \cite{Berman}, \cite{Li}, \cite{LWX}, \cite{LXZ} when $X$ is singular. Despite recent tremendous advances in the field of K-stability, it remains a quite challenging problem to determine whether an explicit Fano variety is K-polystable. We remark that the notion of K-stability also plays a crucial role in moduli theory of Fano varieties. We refer to  \cite{Xu} for a survey of an algebro-geometric aspect of K-stability.

\medskip

It is a classical result that $\nP^3$ admits a K\"{a}hler--Einstein metric. Thus $\Sigma_1(\nP^1, \sO_{\nP^1}(3))=\nP^3$ is K-polystable. Notice that $\Sigma=\Sigma_1(\nP^1, \sO_{\nP^1}(4))$ is a cubic threefold in $\nP^4$ projectively equivalent to $V(x_0x_1x_2+x_3^3+x_4^3) \subseteq \nP^4$, which is K-polystable by \cite[Corollary 1.2]{LX}. The first main result of the paper shows that every first secant variety of a rational normal curve is K-polystable.

\begin{theorem}\label{thm:main1}
Let $\Sigma$ be the first secant variety of a rational normal curve of degree $d \geq 3$. Then $\Sigma$ is K-polystable. Consequently, $\Sigma$ admits a weak K\"{a}hler--Einstein metric.
\end{theorem}

Recall from \cite[Theorem 1.6]{LXZ} that $\Sigma$ is K-polystable if and only if $\Sigma$ admits a weak K\"{a}hler--Einstein metric. An important ingredient of the proof of Theorem \ref{thm:main1} is that if $d \geq 4$, then
$$
\Aut(\Sigma) = \Aut(\nP^1)= \PGL(2, \C)~~\text{ (Proposition \ref{prop:Aut(Sigma)})}.
$$ 
As $\Aut(\Sigma)$ is infinite, $\Sigma$ cannot be K-stable by \cite[Corollary 1.3]{BX}. This means that it is almost impossible to check the K-polystability of $\Sigma$ by estimating the delta invariant $\delta(\Sigma)$ as was done in many recent studies in the field of K-stability. Recall from \cite[Theorem B]{BJ} (see also \cite{FO}) that a Fano variety is K-stable (resp. K-semistable) if and only if $\delta(X) > 1$ (resp. $\delta(X) \geq 1$). Granting Theorem \ref{thm:main1}, we obtain
$$
\delta(\Sigma) = 1.
$$
There is no way to capture K-polystability of $\Sigma$ by only estimating $\delta(\Sigma)$. Instead, one may use $\Aut(\Sigma)$-equivariant alpha invariant $\alpha_{\Aut(\Sigma)}(\Sigma)$. It was shown by Tian \cite{Tian0} that if $G$ is a reductive algebraic group acting on a Fano variety $X$ of dimension $n$ and $\alpha_G(X) > n/(n+1)$, then $X$ is K-polystable (see \cite[Corollary 4.15]{Zhuang}). This approach was very successful to prove that Fano threefolds of degree 22 with infinite automorphism group are K-polystable (see \cite{CS}, \cite{Donaldson2}, \cite{Fujita2}). The family $\{V_u\}_{u \in \C \setminus \{0,1\}}$ of Fano threefolds of degree 22 with infinite automorphism group are parametrized by $\C \setminus \{0, 1\}$. The Mukai--Umemura threefold is $V_{-1/4}$. Note that $\Aut(V_{-1/4})= \PGL(2, \C)$ as in our case. Donaldson \cite{Donaldson2} proved that $V_{-1/4}$ is K-polystable by computing $\alpha_{\Aut(V_{-1/4})}=5/6$. Note that $\Aut(V_u)=\C^* \rtimes (\Z/2\Z)$ unless $u=-1/4$. Cheltsov--Shramov \cite[Theorem 1.5]{CS} completely calculated $\alpha_{\Aut(V_u)}$ for all $u \in \C \setminus \{0,1\}$. In particular, we have $\alpha_{\Aut(V_u)} > 3/4$ (hence $V_u$ is K-polystable) except when $u=3/4$ or $2$. Fujita \cite[Theorem 1.2]{Fujita2} finally verified that the two exceptional cases are also K-polystable. However, in our case, estimating $\Aut(\Sigma)$-equivariant alpha invariant $\alpha_{\Aut(\Sigma)}(\Sigma)$ does not yield that $\Sigma$ is K-polystable. Indeed, 
$$
\alpha_{\Aut(\Sigma)}(\Sigma) = \frac{d+2}{2d}~~\text{ (Remark \ref{rem:G-alpha})},
$$
so $\alpha_{\Aut(\Sigma)}(\Sigma) \leq 3/4$ as soon as $d \geq 4$.

\medskip

To prove Theorem \ref{thm:main1}, we utilize Zhuang's result \cite[Corollary 4.14]{Zhuang} (see Theorem \ref{thm:Zhuang}): If $G$ is a reductive algebraic group acting on a Fano variety $X$ and $S_X(F) < A_X(F)$ for every $G$-invariant irreducible divisor $F$ over $X$, then $X$ is K-polystable. Here
$$
S_X(F):=\frac{1}{(-K_X)^{\dim X}} \int_0^{\infty} \vol(-K_X - xF) dx,
$$
and $A_X(F)$ is the log discrepancy of $F$. Notice that $\PGL(2, \C)$ is a reductive algebraic group. To apply Zhuang's result, we classify $\Aut(\Sigma)$-orbits (Proposition \ref{prop:Aut-orbits}). This reduces the problem to checking
\begin{equation}\label{eq:intro}\tag{\text{\footnotesize $\bigstar$}}
\begin{array}{l}
\text{$S_\Sigma(F) < A_\Sigma(F)$ for every $G$-invariant irreducible divisor $F$ over $\Sigma$}\\
\text{whose center on $\Sigma$ is the rational curve $C$.}
\end{array}
\end{equation}
Taking an explicit log resolution of the log pair $(\Sigma, T)$, we can compute the log canonical threshold of $T$:
$$
\lct(\Sigma, T)=\frac{d+2}{2d}~~\text{ (Lemma \ref{lem:lct})}.
$$ 
One may say that $T$ has not too bad singularities. Let $D$ be an effective divisor $\Q$-linearly equivalent to $-K_{\Sigma}$ with $T \not\subseteq \Supp(D)$. Note that $(\Sigma, D)$ may not be a log canonical pair (see Remark \ref{rem:alpha}). However, we can argue that the order of vanishing of $D$ along the rational normal curve $C$ behaves as if $(\Sigma, D)$ were a log canonical pair (Lemma \ref{lem:ord_FD}). This is a crucial step of the proof of Theorem \ref{thm:main1}. Then a variant of Fujita's result \cite[Proposition 3.2]{Fujita2} (see Proposition \ref{prop:Fujita}) implies (\ref{eq:intro}), so completes the proof. 

\medskip

Next, we study anticanonical polar cylinders in secant varieties of rational normal curves. Recall that a \emph{$(-K_X)$-polar cylinder} on a $\Q$-Fano variety $X$ is a Zariski open subset $U:=X \setminus \operatorname{Supp}(D)$ for an effective divisor $D \sim_{\Q} -K_X$ such that $U \cong \C^1 \times Z$ for an affine variety $Z$. If $X$ contains a $(-K_X)$-polar cylinder, then $X$ is a rational variety and furthermore 
$$
V:=\Spec \Big( \bigoplus_{m \geq 0} H^0(X, \sO_X(-mkK_X)) \Big)
$$ 
admits an effective $\mathbf{G}_a$-action, where $k>0$ is an integer such that $-kK_X$ is a Cartier divisor (\cite[Theorem 1.15]{CPPZ}; see also \cite{CD}, \cite{KPZ}). There has been a considerable amount of interest to determine whether a given Fano variety has an anticanonical polar cylinder (e.g., \cite{Cheltsov2016,Cheltsov2017}). We refer to  \cite{CPPZ} for a survey of cylinders in Fano varieties.

\medskip

It is easy to check that $\nP^3=\Sigma_1(\nP^1, \sO_{\nP^1}(3))$ has a $(-K_{\nP^3})$-polar cylinder. Indeed, $\nP^3 \setminus D \cong \C^1 \times \C^2$, where $D \sim_{\Q} -1/4K_{\nP^3}$ is a hyperplane of $\nP^3$. On the other hand, there is no $(-K_X)$-polar cylinder in a smooth cubic threefold $X$ because $X$ is not rational \cite{CG}. However, the second main result of the paper shows that every first secant variety of a rational normal curve admits an anitcanonical polar cylinder. In particular, the degenerated cubic threefold $\Sigma_1(\nP^1, \sO_{\nP^1}(4))=V(x_0x_1x_2 + x_3^3+x_4^3) \subseteq \nP^4$ admits an anticanonical polar cylinder. 

\begin{theorem}\label{thm:main2}
Let $\Sigma$ be the first secant variety of a rational normal curve of degree $d \geq 3$. Then there exists a $(-K_{\Sigma})$-polar cylinder in $\Sigma$. Consequently, the affine cone in $\mathbf{A}^{d+1}$ over $\Sigma$ in $\mathbf{P}^d$ admits an effective $\mathbf{G}_a$-action.
\end{theorem}

Recall from \cite[Corollary 1.17]{CPPZ} that if $X \subseteq \nP^r$ is a Fano variety of Picard number 1, then the affine cone in $\mathbf{A}^{r+1}$ over $X$ admits an effective $\mathbf{G}_a$-action if and only if $X$ is cylindrical. To prove Theorem \ref{thm:main2}, we directly construct an effective $\Q$-divisor $D \sim_{\Q} -K_{\Sigma}$ such that 
\begin{equation}\label{eq:intro2}\tag{\text{\footnotesize $\spadesuit$}}
\Sigma \setminus \operatorname{Supp}(D) \cong \C^1 \times (\C^1 \times \C^*).
\end{equation}
If $\beta \colon B \to \Sigma$ is the blow-up of $\Sigma$ along $C$, then $B=\nP(E)$ for some rank two vector bundle on $\nP^2$ (see \cite[Theorem 1.1]{ENP2}). Let $\pi \colon B \to \nP^2$ be the canonical projection. Then $B \setminus \pi^{-1}(\ell) \cong \C^2 \times \nP^1$ for any line $\ell$ in $\nP^2$. For some specific choice of $\ell$, we explicitly describe $H_0 \cap B \setminus \pi^{-1}(\ell)$ in $\C^2 \times \nP^1$ for some tautological divisor $H_0$ on $B=\nP(E)$, and we show that
$$
\Sigma \setminus (\beta(\pi^{-1}(\ell)) \cup \beta(H_0)) = B \setminus ( \pi^{-1}(\ell) \cup  H_0) \cong \C^1 \times \C^1 \times \C^*.
$$
Note that $\beta(\pi^{-1}(\ell)) \sim_{\Q} -1/2 K_{\Sigma}$ and $\beta(H_0) \sim_{\Q} -(d-2)/4 K_{\Sigma}$. Thus the effective divisor $D:=4/d(\beta(\pi^{-1}(\ell))  + \beta(H_0))$ satisfies (\ref{eq:intro2}).

\medskip

It is expected by Choe--Kwak \cite{CK} that there would be a “matryoshka structure” among all secant varieties of a fixed projective variety. This means that similar patterns of algebraic/geometric properties of secant varieties repeat over and over again as matryoshka dolls. It was shown in \cite{CKP}, \cite{ENP1} that syzygies of secant varieties of smooth projective curve have a matryoshka structure. The \emph{$k$-th secant variety} 
$$
\Sigma_k:=\Sigma_k(\nP^1, \sO_{\nP^1}(d)) \subseteq \nP^d
$$ 
of a rational normal curve $C \subseteq \nP^d$ of degree $d$ is the closure of the union of $(k+1)$-secant $k$-planes to $C$ in $\nP^d$. If $d \geq 2k+1$, then \cite[Theorem 1.1]{ENP1} says that $\Sigma_k$ is a $\Q$-factorial normal projective variety of Picard number 1 such that $\Sigma_k$ has log terminal singularities and $-K_{\Sigma}$ is ample. Note that $\dim \Sigma_k = 2k+1$ and $\Sing \Sigma_k = \Sigma_{k-1}$ unless $d=2k+1$ (in this case $\Sigma_k = \nP^{2k+1}$). From the point of view of the matryoshka structure, it is natural to expect that the main results of the present paper would be extended to higher secant varieties of rational normal curves.

\begin{conjecture}
Let $\Sigma_k$ be the $k$-th secant variety of a rational normal curve of degree $d \geq 2k+1$. Then $\Sigma_k$ is K-polystable, and there is a $(-K_{\Sigma_k})$-polar cylinder in $\Sigma_k$. 
\end{conjecture}

The paper is organized as follows. We begin with recalling basic facts on the first secant variety $\Sigma$ of a rational normal curve in Section \ref{sec:secvar}, where we also study the automorphism group of $\Sigma$. Section \ref{sec:K-poly} is devoted to proving K-polystability of $\Sigma$ (Theorem \ref{thm:main1}). Finally, in Section \ref{sec:cyl}, we establish that $\Sigma$ admits a $(-K_{\Sigma})$-polar cylinder (Theorem \ref{thm:main2}).

\section{Secant Varieties}\label{sec:secvar}
In this section, we collect relevant basic facts on secant varieties and tangent developable surfaces of rational normal curves. We refer to \cite{ENP1} and \cite{Park} for more details. We keep using the notations in the introduction. Let 
$$
C \subseteq \nP H^0(\nP^1, \sO_{\nP^1}(d)) =\nP^d
$$
be a rational normal curve of degree $d \geq 4$. In this paper, the projectivization $\nP V$ of a finite dimensional vector space $V$ over $\C$ is the space of one-dimensional quotient of $V$. Let
$$
\Sigma:=\Sigma_1(\nP^1, \sO_{\nP^1}(d)) \subseteq  \nP^d
$$
be the first secant variety of $C$ in $\nP^d$. Recall that $\Sigma$ is a $\Q$-factorial normal projective variety of Picard number 1 such that $\Sigma$ has log terminal singularities and $-K_{\Sigma}$ is ample (\cite[Theorem 1.1]{ENP1}). 

\medskip

We freely use the standard definitions of singularities of the minimal model program  in \cite{Kollar}, \cite{KM}. Here we briefly recall some basic necessary definitions. Let $(X, \Gamma)$ be a log pair. In other words, $X$ is a normal projective variety, and $\Gamma$ is an effective $\Q$-divisor such that $K_X + \Gamma$ is $\Q$-Cartier. For a prime divisor $F$ over $X$, take a birational morphism $f \colon \widetilde{X} \to X$ such that $F$ is a divisor on $\widetilde{X}$. The \emph{log discrepancy} of $F$ is defined as
$$
A_X(F):=1+\ord_F (K_{\widetilde{X}} - f^*K_X).
$$
We say that $(X, \Gamma)$ is a \emph{klt} (\emph{Kawamata log terminal}) pair (resp. an \emph{lc} (\emph{log canonical}) pair) if $A_X(F) > 0$ (resp. $A_X(F) \geq 0$) for every prime divisor $F$ over $X$. We say that $X$ has \emph{log terminal singularities} if $(X, 0)$ is a klt pair. 

\medskip

Regarding $\nP^2$ as the Hilbert scheme of two points on $\nP^1$, we have the universal family map 
$$
\sigma \colon \nP^1 \times \nP^1 \longrightarrow \nP^2,~~(x,y) \longmapsto x+y.
$$
Consider the following globally generated rank two vector bundle on $\nP^2$ defined by
$$
E:=\sigma_* (\sO_{\nP^1} \boxtimes \sO_{\nP^1}(d)).
$$ 
Note that $\det E = \sO_{\nP^2}(d-1)$. Put $B:=\nP(E)$, and let $\pi \colon \nP(E) \to \nP^2$ be the canonical projection. The tautological divisor $H$ on $B$ with $\sO_B(H) = \sO_{\nP(E)}(1)$ is a base point free, and the image of the morphism given by $|H|$ is the secant variety $\Sigma$. We denote by the map
$$
\beta \colon B \longrightarrow \Sigma,
$$
which is the blow-up of $\Sigma$ along $C$ with the exceptional divisor $Z=\nP^1 \times \nP^1$ (\cite[Theorem 1.1]{ENP2}). Then $\beta$ is a resolution of singularities of $\Sigma$. Notice that $\pi|_Z = \sigma \colon \nP^1 \times \nP^1 \to \nP^2$. We have a commutative diagram
$$
\xymatrix{
Z=\nP^1 \times \nP^1 \ar@{^{(}->}[r]  \ar[rd]_-{\sigma} & B=\nP(E) \ar[r]^-{\beta}  \ar[d]^-{\pi} & \Sigma \subseteq \nP^d \\
& \nP^2. &
}
$$
Let $A$ be a line on $\nP^2$ so that $\sO_{\nP^2}(A) = \sO_{\nP^2}(1)$. By \cite[Proposition 3.15]{ENP1}, we have
$$
Z \sim 2H - \pi^* (d-2)A.
$$
Since $K_B = -2H + \pi^* (d-4)A$, it follows that
$$
K_B+Z+\pi^* 2A \sim 0.
$$
By abuse of notation, denote by $H$ a hyperplane section of $\Sigma \subseteq \nP^d$ so that $\beta^* H = H$. 

\begin{proposition}\label{prop:-K_{Sigma}}
We have
$$
-K_{\Sigma} \sim_{\Q} \frac{4}{d-2}H~~\text{ and }~~K_B = \beta^* K_{\Sigma} - \frac{d-4}{d-2} Z.
$$
\end{proposition}

\begin{proof}
As $\Sigma$ is $\Q$-factorial and the Picard number of $\Sigma$ is $1$, we may write $-K_{\Sigma} \sim_{\Q} a H$ for some positive $a \in \Q$. We have
$$
K_B = \beta^* K_{\Sigma} - b Z
$$
where $b:=A_{\Sigma}(Z)-1$ is the discrepancy of $Z$. Then 
$$
-2H + \pi^*(d-4)A = -aH - 2bH + b(d-2) \pi^*A,
$$
so $a=4/(d-2)$ and $b=(d-4)/(d-2)$.
\end{proof}

As $H^3 = \deg \Sigma = (d-1)(d-2)/2$ (\cite[Proposition 5.11]{ENP1}), we get
$$
(-K_{\Sigma})^3 = 32\frac{(d-1)}{(d-2)^2}.
$$

\begin{remark}
Let $\Sigma_k$ be the $k$-th secant variety of a rational normal curve of degree $d \geq 2k+1$ in $\nP^d$, and $H_k$ be a hyperplane section. Then one can also show that
$$
-K_{\Sigma_k} \sim_{\Q}  \frac{2(k+1)}{d-2k}H_k.
$$
As $H_k^{2k+1} = \deg \Sigma_{k,d}= { d -k \choose k+1}$ (\cite[Proposition 5.11]{ENP1}), we get
$$
(-K_{\Sigma_k})^{2k+1} =  \left( \frac{2(k+1)}{d-2k} \right)^{2k+1} { d -k \choose k+1}.
$$
It is also known that if $\beta_k \colon B_k \to \Sigma_k$ is the blow-up of $\Sigma_k$ along $\Sigma_{k-1}$, then $B_k =\nP(E)$ for some globally generated rank $k+1$ vector bundle on $\nP^{k+1}$ (\cite[Theorem 1.1]{ENP2}). Let $Z_{k-1}$ be the exceptional divisor of $\beta_k$. Then one can also prove that
$$
K_{B_k}=\beta_k^*K_{\Sigma_{k}} - \frac{d-2k-2}{d-2k}Z_{k-1}.
$$
\end{remark}

Now, let $\Delta \subseteq \nP^1 \times \nP^1$ be the diagonal. Then $Q:=\sigma(\Delta) \subseteq \nP^2$ is a smooth conic. Notice that $E|_Q$ is the first jet bundle of $\sO_{\nP^1}(d)$. By \cite[Corollary 1.8]{Kaji} (see also \cite[Subsection 2.5]{Park}), 
$$
E|_Q = \sO_{\nP^1}(d-1) \oplus \sO_{\nP^1}(d-1).
$$ 
Put
$$
\widetilde{T}:=\pi^{-1}(Q)=\nP (E|_Q) = \nP^1 \times \nP^1.
$$
Then $\pi_{\widetilde{T}}:=\pi|_{\widetilde{T}} \colon \nP (E|_Q) \to Q$ is the canonical projection. Observe that $\widetilde{T}$ and $Z$ meet along the diagonal with multiplicity $2$. 
Note that
$$
T:=\beta(\widetilde{T}) \subseteq \nP^d
$$
is the tangent developable surface of $C$ in $\nP^d$. 

\begin{proposition}\label{prop:beta^*T}
We have
$$
T \sim -K_{\Sigma}~~\text{ and }~~\beta^* T =\widetilde{T}+ \frac{2}{d-2} Z.
$$
\end{proposition}

\begin{proof}
Since $K_B+Z + \widetilde{T} \sim 0$, it follows that $T \sim -K_{\Sigma}$. We know that $T \sim 4/(d-2) H$ (see Proposition \ref{prop:-K_{Sigma}}). We have $\beta^* T =\widetilde{T}+ c Z$. Then 
$$
\frac{4}{d-2} H = 2\pi^*A + 2cH - c(d-2)\pi^*A,
$$
so $c=2/(d-2)$.
\end{proof}

Next, we turn to the study of automorphism group of $\Sigma$. The next two propositions are important ingredients of the proof of Theorem \ref{thm:main1}.

\begin{proposition}\label{prop:Aut(Sigma)}
$\Aut(\Sigma) = \Aut(\nP^1) = \PGL(2, \C)$.
\end{proposition}

\begin{proof}
Since every automorphism of $\Sigma$ preserves $C=\Sing \Sigma$, the restriction map
$$
\gamma \colon \Aut(\Sigma) \longrightarrow \Aut(C),~~\varphi \longmapsto \varphi|_C
$$
is a well-defined group homomorphism. We show that $\gamma$ is an isomorphism. First, we check that $\gamma$ is injective. Take any $\varphi \in \ker(\gamma)$ so that $\varphi|_C = \operatorname{id}_C$, i.e.,  $\varphi(x) = x$ for all $x \in C$. It suffices to confirm that $\varphi(x)=x$ for any $x \in \Sigma \setminus C$. We can find a unique line $\ell$ in $\nP^d$ passing through $x$ such that $\ell$ is either a bi-secant line to $C$ or a tangent line to $C$. Then $\varphi(\ell)$ is a bi-secant line to $C$ or a tangent line to $C$, respectively. Since $\varphi|_C = \operatorname{id}_C$, it follows that $\ell \cap C=\varphi(\ell) \cap C $. Thus $ \ell = \varphi(\ell)$. Now, we can choose a hyperplane section $H$ of $\Sigma \subseteq \nP^d$ such that $H \cap \ell = \{ x\}$. Since $\varphi|_C = \operatorname{id}_C$, it follows that $H \cap C = \varphi(H) \cap C$. Thus $H = \varphi(H)$. Then we find
$$
\{\varphi(x)\} = \varphi(H \cap \ell) = \varphi(H) \cap \varphi(\ell) = H \cap \ell = \{x\}.
$$
Hence $\gamma$ is injective.

\medskip

Next, we check that $\gamma$ is surjective. The group action of $\Aut(\nP^1)$ on $\nP^1$ naturally induces a group action of $\Aut(\nP^1)$ on $\nP H^0(\nP^1, \sO_{\nP^1}(d))=\nP^d$. More precisely, the natural group action of $\operatorname{GL}(2, \C)$ on $H^0(\nP^1, \sO_{\nP^1}(1))$ uniquely determines the given group action of $\PGL(2, \C)$ on $\nP^1=\nP H^0(\nP^1, \sO_{\nP^1}(1))$. Then $\operatorname{GL}(2, \C)$ acts on $H^0(\nP^1, \sO_{\nP^1}(d))$ in such a way that $\varphi(u^{d-k}v^k) = \varphi(u)^{d-k}\varphi(v)^k$ for any $\varphi \in \operatorname{GL}(2, \C)$, where $u,v$ form a basis of $H^0(\nP^1, \sO_{\nP^1}(1))$. This group action induces a group action of $\PGL(2, \C)$ on $\nP H^0(\nP^1, \sO_{\nP^1}(d))$. In this way, we may regard $\Aut(\nP^1)$ as a subgroup $G$ of $\Aut(\nP^d)$ preserving the rational normal curve $C$ of degree $d$ in $\nP^d$. The restriction group homomorphism
$$
\mu \colon G \longrightarrow \Aut(C),~~\varphi \longmapsto \varphi|_C
$$
is an isomorphism. Every $\varphi \in G$ sends a bi-secant line (resp. a tangent line) to $C$ in $\nP^d$ to a bi-secant line (resp. a tangent line) to $C$ in $\nP^d$. This means that $\varphi|_{\Sigma} \in \Aut(\Sigma)$. As $\gamma(\varphi|_{\Sigma}) = \varphi|_C$, we can conclude that $\gamma$ is surjective. 
\end{proof}

\begin{proposition}\label{prop:Aut-orbits}
There are exactly three $\Aut(\Sigma)$-orbits: $C$, $T \setminus C$, and $\Sigma \setminus T$.
\end{proposition}

\begin{proof}
Clearly, every automorphism of $\Sigma$ preserves $T$ and $C$. As $\Aut(\Sigma)=\Aut(C)$ can permute both sets $\{\text{length two subschemes of $C$}\}$ and $\{\text{points of $C$}\}$, it suffices to show that if $\ell$ is a bi-secant line or a tangent line to $C$ and $x,y \in \ell \setminus C$, then $\varphi(x)=y$ for some $\varphi \in \Aut(\Sigma)$ with $\varphi(\ell)=\ell$. Recall that the group action of $\Aut(\Sigma)=\PGL(2,\C)$ on $\Sigma$ extends to a group action on $\nP^d=\nP H^0(\nP^1, \sO_{\nP^1}(d))$ which is induced from a group action of $\operatorname{GL}(2, \C)$ on $H^0(\nP^1, \sO_{\nP^1}(d))$. Let $u,v$ be a basis of $H^0(\nP^1, \sO_{\nP^1}(1))$. We may think that the group action of $\operatorname{GL}(2, \C)$ on $H^0(\nP^1, \sO_{\nP^1}(d))$ is given by
$$
\varphi(u^{d-k}v^k) = (\alpha u+ \gamma v)^{d-k} (\beta u+\delta v)^k~~\text{for}~~\varphi = \begin{bmatrix} \alpha & \beta \\ \gamma & \delta \end{bmatrix} \in \operatorname{GL}(2, \C).
$$
We may regard $u,v$ as coordinates of $\nP^1$. 

\medskip

Let $\xi:=C \cap \ell$ be the scheme-theoretic intersection. Then $\ell = \nP H^0(\xi, \sO_{\nP^1}(d)|_{\xi})$. Via the natural restriction map
$$
H^0(\nP^1, \sO_{\nP^1}(d)) \longrightarrow H^0(\xi, \sO_{\nP^1}(d)|_{\xi}),
$$
every element of $\Aut(\Sigma)$ gives an isomorphism between bi-secant or tangent lines
$$
\ell=\nP H^0(\xi, \sO_{\nP^1}(d)|_{\xi})  \longrightarrow \nP H^0(\varphi(\xi), \sO_{\nP^1}(d)|_{\varphi(\xi)})=\varphi(\ell).
$$

\medskip

Suppose that $\ell$ is a bi-secant line to $C$. We may assume that $\xi=[1,0]+[0,1]$. Then $u^d|_{\xi}, v^d|_{\xi}$ form a basis of $H^0(\xi, \sO_{\nP^1}(d)|_{\xi})$. If
$$
\varphi = \begin{bmatrix} 1 & 0 \\ 0 & a \end{bmatrix} ~~\text{for any}~~a \in \C^*,
$$
then $\varphi(\ell)=\ell$. Such $\varphi$ form a subgroup $K$ of $\operatorname{GL}(2, \C)$ which acts on $H^0(\xi, \sO_{\nP^1}(d)|_{\xi})$ by $\varphi(u^d|_{\xi}) = u^d|_{\xi}$ and $\varphi(v^d|_{\xi}) = a^d v^d|_{\xi}$.  This shows that $\ell \setminus \xi$ is a $K$-orbit, where we regard $K$ as a subgroup of $\PGL(2,\C)=\Aut(\Sigma)$.

\medskip

Suppose that $\ell$ is a tangent line to $C$. We may assume that $\xi=2[1,0]$. Then $u^d|_{\xi}, u^{d-1}v|_{\xi}$ form a basis of $H^0(\xi, \sO_{\nP^1}(d)|_{\xi})$. If
$$
\varphi = \begin{bmatrix} 1 & a \\ 0 & b \end{bmatrix} ~~\text{for any}~~a \in \C,~~b \in \C^*,
$$
then $\varphi(\ell)=\ell$. Such $\varphi$ form a subgroup $K'$ of $\operatorname{GL}(2, \C)$ which acts on $H^0(\xi, \sO_{\nP^1}(d)|_{\xi})$ by $\varphi(u^d|_{\xi}) = u^d|_{\xi}$ and $\varphi(u^{d-1}v|_{\xi}) = (au^d + bu^{d-1}v)|_{\xi}$.  This shows that $\ell \setminus \xi$ is a $K'$-orbit, where we regard $K'$ as a subgroup of $\PGL(2,\C)=\Aut(\Sigma)$.
\end{proof}

\section{K-polystability}\label{sec:K-poly}
In this section, we prove Theorem \ref{thm:main1}. We keep using the notations introduced in the previous sections. It is enough to show that the secant variety $\Sigma$ of a rational normal curve of degree $d \geq 4$ is K-polystable. For this purpose, we utilize Zhuang's result.

\begin{theorem}[{\cite[Corollary 4.14]{Zhuang}}]\label{thm:Zhuang}
Let $G$ be an algebraic group, and $X$ be a Fano variety with a $G$-action. If $G$ is reductive and $S_X(F) < A_X(F)$ for every $G$-invariant irreducible divisor $F$ over $X$, then $X$ is K-polystable.
\end{theorem}

Here we briefly recall basic necessary definitions. Let $X$ be a \emph{Fano variety}, which means that $X$ has log terminal singularities and $-K_X$ is ample. Recall that
$$
S_X(F):=\frac{1}{(-K_X)^{\dim X}} \int_0^{\infty} \vol(-K_X - xF) dx,
$$
and $A_X(F)$ is the log discrepancy of $F$. For any divisor $D$ on $X$ with $\dim X = n$, we define the \emph{volume} of $X$ as
$$
\vol (D)=\vol_X(D):=\limsup_{m \to \infty} \frac{h^0(X, \lfloor mD \rfloor)}{m^n/n!}.
$$
When $F$ is a divisor on $\widetilde{X}$, where $f \colon \widetilde{X} \to X$ is a birational morphism, we set
$$
\vol(-K_X - xF):=\vol_{\widetilde{X}}(f^*(-K_X)-xF).
$$
The \emph{pseudoeffective threshold} of $F$ is defined as
$$
\tau_X(F):=\sup\{ \tau \in \mathbf{R}_{>0}| \vol (-K_X - \tau F)>0 \}.
$$
The \emph{alpha invariant} of $X$ is
$$
\alpha(X):=\inf \{\lct(X, D) \mid -K_X \sim_{\Q} D \geq 0\} = \inf_F \frac{A_X(F)}{\tau_X(F)},
$$
where the infimum takes over prime divisors $F$ over $X$ (\cite[Theorem C]{BJ}). Here the \emph{log canonical threshold} of $(X, D)$ is defined as
$$
\lct(X, D):=\inf\{ c>0 \mid \text{$(X, cD)$ is a log canonical pair}\}.
$$
The \emph{delta invariant} of $X$ is
$$
\delta(X):=\inf_F \frac{A_X(F)}{S_X(F)},
$$
where the infimum takes over prime divisors $F$ over $X$ (\cite[Theorem C]{BJ}). When an algebraic group $G$ acts on $X$, one can also define the $G$-equivariant alpha invariant $\alpha_G(X)$ and the $G$-equivariant delta invariant $\delta_G(X)$ using $G$-invariant divisors, (see \cite{Zhuang} for more details).

\medskip

In view of Theorem \ref{thm:Zhuang}, we need to check that $S_{\Sigma}(F) < A_{\Sigma}(F)$ for every $\Aut(\Sigma)$-invariant  irreducible divisor $F$ over $\Sigma$. Proposition \ref{prop:Aut-orbits} says that if $F$ is an $\Aut(\Sigma)$-invariant  irreducible divisor  over $\Sigma$, then $F=T$ or $c_{\Sigma}(F)=C$, where $c_{\Sigma}(F)$ is the center of $F$ on $\Sigma$. The former case is easy to confirm.

\begin{lemma}\label{lem:A(T)S(T)}
$S_{\Sigma}(T)=1/4$ and $A_{\Sigma}(T)=1$.
\end{lemma}

\begin{proof}
We have $S_{\Sigma}(T)=\int_0^1 (1-x)^3 dx = 1/4$. Clearly, $A_{\Sigma}(T)=1$.
\end{proof}

Henceforth, we focus on the latter case: $c_{\Sigma}(F)=C$. For this case, we use a variant of Fujita's result.

\begin{proposition}[{cf. \cite[Proposition 3.2]{Fujita2}}]\label{prop:Fujita}
Let $X$ be an $n$-dimensional Fano variety, $\eta \in X$ be a scheme-theoretic point, $F$ be an irreducible divisor over $X$ with $c_X(F)=\overline{\eta}$, and $0 < t \leq s$ be positive real numbers. Assume the following:\begin{enumerate}
\item There is a prime divisor $D_0$ on $X$ with $D_0 \sim_{\Q} -K_X$ such that $\lct(X, D_0)_{\eta} \geq t$.
\item $\ord_F D \leq s^{-1} A_X(F)$ for any effective $\Q$-divisor $D \sim_{\Q} -K_X$ with $D_0 \not\subseteq \Supp (D)$.
\end{enumerate}
Then we have
$$
S_X(F) \leq \frac{A_X(F)}{n+1} \left( \frac{n-1}{s} + \frac{1}{t} \right).
$$
\end{proposition}

\begin{proof}
Even though this was essentially shown in \cite[Proof of Proposition 3.2]{Fujita2}, we include the whole proof for reader's convenience. By Condition $(1)$, $\ord_F D_0 \leq t^{-1} A_X(F)$. First, assume that $\ord_F D_0 \leq s^{-1} A_X(F)$. Then, together with Condition $(2)$, we get $\ord_F D \leq s^{-1} A_X(F)$ for every effective $\Q$-divisor $D \sim_{\Q} -K_X$. Thus $s \leq A_X(F)/\tau_X(F)$.
By \cite[Proposition 2.1]{Fujita1}, we obtain
$$
S_X(F) \leq \frac{n}{n+1} \tau_X(F) \leq \frac{n}{n+1} s^{-1} A_X(F) \leq \frac{A_X(F)}{n+1} \left( \frac{n-1}{s} + \frac{1}{t} \right).
$$
Next, assume that $\ord_F D_0 > s^{-1} A_X(F)$. Take a log resolution $f \colon \widetilde{X} \to X$ of $X$ such that $F$ is a divisor on $\widetilde{X}$. For any effective $\Q$-divisor $D \sim_{\Q} -K_X$, there is $e \in [0,1] \cap \Q$ such that $D=eD_0 + (1-e)D'$ with $-K_X \sim_{\Q} D' \geq 0$ and $D_0 \not\subseteq \Supp(D')$. By Condition $(2)$, we have $\ord_F D' \leq s^{-1} A_X(F)$. Choose $x \in (s^{-1}A_X(F), \ord_F D_0) \cap \Q$ such that $\ord_F D \geq x$. Since
$$
x \leq \ord_F D \leq e \ord_F D_0 + (1-e)s^{-1}A_X(F),
$$
it follows that
$$
d \geq \frac{x-s^{-1}A_X(F)}{\ord_F D_0 - s^{-1}A_X(F)}.
$$
This implies that $|-mK_X - mxF|$ has a fixed divisor
$$
m \left( \frac{x-s^{-1}A_X(F)}{\ord_F D_0 - s^{-1}A_X(F)} \right) D_0
$$
for a sufficiently divisible integer $m > 0$. Thus we have
$$
\begin{array}{l}
\vol(-K_X - xF)\\[10pt]
\displaystyle = \vol \left( -f^*K_X - xF - \frac{x-s^{-1}A_X(F)}{\ord_F D_0 - s^{-1}A_X(F)} (f^* D_0 - \ord_F D_0)F \right)\\[10pt]
\displaystyle  = \left( \frac{\ord_F D_0 - x}{\ord_F D_0 - s^{-1}A_X(F)}\right)^n \vol(-K_X - s^{-1}A_X(F)F).
\end{array} 
$$
By \cite[Proposition 3.1]{Fujita2}, we obtain
$$
S_X(F) \leq \frac{1}{n+1} \left(\frac{n-1}{s} A_X(F) + \ord_F D_0 \right) \leq \frac{A_X(F)}{n+1} \left( \frac{n-1}{s} + \frac{1}{t} \right).
$$
This completes the proof.
\end{proof}

The next two lemmas allow us to confirm $S_{\Sigma}(F) < A_{\Sigma}(F)$ for every $\Aut(\Sigma)$-invariant irreducible divisor $F$ over $\Sigma$ with $c_\Sigma(F)=C$ by applying Proposition \ref{prop:Fujita}. To this end, we construct a log resolution of the log pair $(\Sigma, T)$. Recall that $\widetilde{T}=\beta_*^{-1} T$ and $Z$ meet along the diagonal $\Delta=\nP^1$ with multiplicity $2$. Let $\beta_1 \colon B_1 \to B$ be the blow-up of $B$ along $\Delta$ with the exceptional divisor $E_1$. Then $\widetilde{T}_1:=\beta_{1,*}^{-1} \widetilde{T}, Z_1:=\beta_{1,*}^{-1} Z, E_1$ meet along a curve $\Delta_1=\nP^1$. Let $\beta_2 \colon B_2 \to B_1$ be the blow-up of $B_1$ along $\Delta_1$ with the exceptional divisor $E_2$. Then the union of $\widetilde{T}_2:=\beta_{2,*}^{-1} \widetilde{T}_1, Z_2:=\beta_{2,*}^{-1} Z_2, \widetilde{E}_1:=\beta_{2,*}^{-1} E_1, E_2$ is simple normal crossing, so  $f:=\beta_2 \circ \beta_1 \circ \beta$ is a log resolution of the log pair $(\Sigma, T)$. 

\begin{lemma}\label{lem:lct}
$\lct(\Sigma, T) = (d+2)/(2d)$.
\end{lemma}

\begin{proof}
We have
$$
\begin{array}{lcl}
K_B = \beta^* K_{\Sigma} - \frac{d-4}{d-2} Z ~\text{(Proposition \ref{prop:-K_{Sigma}})}& \text{and} & \beta^*T = \widetilde{T} + \frac{2}{d-2} Z~\text{(Proposition \ref{prop:beta^*T})};\\[10pt]
 K_{B_1} = (\beta_1 \circ \beta)^* K_{\Sigma} - \frac{d-4}{d-2} Z_1 + \frac{2}{d-2} E_1 & \text{and} & (\beta_1 \circ \beta)^* T = \widetilde{T}_1 + \frac{2}{d-2} Z_1 + \frac{d}{d-2} E_1;\\[10pt]
 K_{B_2}=f^* K_{\Sigma} - \frac{d-4}{d-2} Z_2 + \frac{2}{d-2} \widetilde{E}_1 + \frac{4}{d-2} E_2 & \text{and} & f^* T = \widetilde{T}_2 + \frac{2}{d-2} Z_2 + \frac{d}{d-2} \widetilde{E}_1 + \frac{2d}{d-2} E_2.
\end{array}
$$
Then
$$
\lct(\Sigma, T) = \min \left\{\frac{A_{\Sigma}(T)}{\ord_{T} T}, \frac{A_{\Sigma}(Z)}{\ord_{Z} T},  \frac{A_{\Sigma}(E_1)}{\ord_{E_1} T}, \frac{A_{\Sigma}(E_2)}{\ord_{E_2} T} \right\} = \frac{A_{\Sigma}(E_2)}{\ord_{E_2} T} = \frac{1+\frac{4}{d-2}}{\frac{2d}{d-2}} =  \frac{d+2}{2d}.
$$
We finish the proof.
\end{proof}

\begin{lemma}\label{lem:ord_FD}
Let $F$ be an $\Aut(\Sigma)$-invariant irreducible divisor over $\Sigma$ with $c_{\Sigma}(F)=C$. Then we have
$$
\ord_F D \leq  A_{\Sigma}(F)
$$
for any effective $\Q$-divisor $D \sim_{\Q} -K_{\Sigma}$ with $T \not\subseteq \Supp (D)$.
\end{lemma}

\begin{proof}
It suffices to check the assertion when $\Supp (D)$ is irreducible. By Proposition \ref{prop:beta^*T},
$$
\beta^*D \sim_{\Q} \beta^*T \sim_{\Q} \pi^*2A + \frac{2}{d-2}Z,
$$ 
so we may write
$$
\beta^*D = \beta_*^{-1}D + aZ~~\text{for some}~~ 0 \leq a \leq \frac{2}{d-2}.
$$
Then 
$$
\ord_Z D = a \leq \frac{2}{d-2} = 1- \frac{d-4}{d-2} = A_{\Sigma}(Z).
$$
From now on, assume that $F \neq Z$ so that $c_B(F)=\Delta=\nP^1$. Recall that $\widetilde{T}=\beta_*^{-1} T = \nP^1 \times \nP^1$ and $Z=\nP^1 \times \nP^1$ meet along the diagonal $\Delta$ with multiplicity $2$. We have
$$
\beta_*^{-1}D \sim_{\Q} \Big( \frac{4}{d-2}-2a \Big) H + \pi^* a(d-2)A \sim_{\Q} \pi^*2A +  \Big( \frac{2}{d-2}-a \Big)Z,
$$
so we get
$$
(\beta_*^{-1}D )|_Z \sim_{\Q} \Big( \frac{4}{d-2}-2a \Big) H|_Z + a(d-2)\Delta~~\text{ and }~~(\beta_*^{-1}D )|_{\widetilde{T}} \sim_{\Q} \pi_{\widetilde{T}}^* 2A|_{Q} +  \Big( \frac{4}{d-2}-2a \Big) \Delta.
$$
Note that $\sO_Z(H|_Z) = \sO_{\nP^1} \boxtimes \sO_{\nP^1}(d)$ and $\sO_{\widetilde{T}}(\pi_{\widetilde{T}}^* A|_{Q}) = \sO_{\nP^1} \boxtimes \sO_{\nP^1}(2)$. 
Then we have
$$
\ord_{\Delta} (\beta_*^{-1} D)|_Z \leq a(d-2)~~\text{ and }~~ \ord_{\Delta} (\beta_*^{-1} D)|_{\widetilde{T}} \leq \frac{4}{d-2}-2a.
$$
Thus we obtain
$$
\ord_{\Delta} \beta_*^{-1} D \leq  \underbrace{ \min\left\{ a(d-2), \frac{4}{d-2}-2a \right\} \leq \frac{4}{d}}_{\text{the equality holds $\Longleftrightarrow$ $a=\frac{4}{d(d-2)}$}}.
$$

\medskip

Suppose that $\Supp\big((\beta_1 \circ \beta)_*^{-1} D\big)$ does not contain $\Delta_1=\widetilde{T}_1 \cap Z_1 \cap E_1$, which is a unique $\Aut(\Sigma)$-invariant curve on $B_1$. We find
$$
\begin{array}{rcl}
\ord_F D& = & \ord_F \beta_*^{-1}D + a \ord_F Z \\
&=&  \ord_{\Delta} \beta_*^{-1}D \cdot \ord_F E_1+a \ord_F Z\\
&=&  ( \ord_{\Delta} \beta_*^{-1}D  + a) \cdot \ord_F E_1 + a \ord_F Z_1\\
&\leq&\displaystyle \Big( \frac{4}{d} + \frac{2}{d-2} \Big) \ord_F E_1 + \frac{2}{d-2} \ord_F Z_1.
\end{array}
$$
Since
$$
K_{B_1} + \Big( \frac{4}{d} E_1 + Z_1 \Big) = (\beta_1 \circ \beta)^* K_{\Sigma} +  \Big( \frac{4}{d} + \frac{2}{d-2} \Big) E_1 + \frac{2}{d-2} Z_1
$$
and $(B_1, (4/d) E_1 + Z_1)$ is a log canonical pair, it follows that 
$$
 \ord_F D \leq \Big( \frac{4}{d} + \frac{2}{d-2} \Big) \ord_F E_1 + \frac{2}{d-2} \ord_F Z_1 \leq A_{\Sigma}(F).
$$

\medskip

Suppose that $\Supp\big( (\beta_1 \circ \beta)_*^{-1} D\big)$ does contain $\Delta_1$. In this case, notice that
$$
\ord_{\Delta} \beta_*^{-1} D \leq \frac{1}{2} \min\{ \ord_{\Delta} (\beta_*^{-1} D)|_Z,  \ord_{\Delta} (\beta_*^{-1} D)|_{\widetilde{T}}\} \leq \frac{1}{2} \ord_{\Delta} (\beta_*^{-1} D)|_{\widetilde{T}} \leq \frac{2}{d-2}-a.
$$
Then we find
$$
\ord_{\Delta} \Big( \beta_*^{-1}D + aZ+ \frac{d-4}{d-2}Z \Big) \leq \frac{2}{d-2}-a + a + \frac{d-4}{d-2} = 1.
$$
This implies that the log pair 
$$
\left( B, \beta_*^{-1}D + aZ+ \frac{d-4}{d-2} Z \right)
$$
is log canonical at the generic point of $\Delta$. Since
$$
K_B + \beta_*^{-1}D + aZ + \frac{d-4}{d-2} Z = \beta^*(K_{\Sigma}+D),
$$
it follows that $\ord_F D \leq A_{\Sigma}(F)$.
\end{proof}

Finally, we complete the proof of Theorem \ref{thm:main1}.

\begin{proof}[Proof of Theorem \ref{thm:main1}]
 Let $F$ be an $\Aut(\Sigma)$-invariant irreducible divisor over $\Sigma$. In view of Theorem \ref{thm:Zhuang}, we need to check that $S_{\Sigma}(F) < A_{\Sigma}(F)$. Proposition \ref{prop:Aut-orbits} says that $F = T$ or $c_{\Sigma}(F)=C$. In the former case, Lemma \ref{lem:A(T)S(T)} shows $S_{\Sigma}(T) < A_{\Sigma}(T)$. Assume that $c_{\Sigma}(F)=C$. Applying Proposition \ref{prop:Fujita} with $t=(d+2)/2d$ (by Lemma \ref{lem:lct}) and $s=1$ (by Lemma \ref{lem:ord_FD}), we find
$$
S_{\Sigma}(F) \leq \frac{A_{\Sigma}(F)}{4} \left( \frac{2}{1} + \frac{2d}{d+2} \right) = A_{\Sigma}(F) \frac{4d+4}{4(d+2)} < A_{\Sigma}(F).
$$
Thus we finish the proof.
\end{proof}

\begin{remark}\label{rem:alpha}
The original statement of \cite[Proposition 3.2]{Fujita2} replaces Condition (2) in Proposition \ref{prop:Fujita} with the following condition: $\lct(X, D) \geq s$ for any effective $\Q$-divisor $D \sim_{\Q} -K_X$ with $D_0 \not\subseteq \Supp(D)$. To prove Theorem \ref{thm:main1} by applying  the original statement of \cite[Proposition 3.2]{Fujita2}, we need that $\lct(\Sigma, D) \geq 1$ for any effective $\Q$-divisor $D \sim_{\Q} -K_{\Sigma}$ with $T \not\subseteq \Supp(D)$. However, this is not the case. If $\ell$ is a line in $\nP^2$, then $\beta(\pi^{-1}(\ell))$ is a $\Q$-divisor on $\Sigma$ such that
$$
\beta(\pi^{-1}(\ell)) \sim_{\Q} - \frac{1}{2}  K_{\Sigma}~~\text{ and }~~T=\beta(\pi^{-1}(Q)) \not\subseteq \beta(\pi^{-1}(\ell)).
$$
Putting $D:=2\beta(\pi^{-1}(\ell))$, we see that $(\Sigma, cD)$ is not log canonical whenever $c>1/2$. This means that $\lct(\Sigma, D) \leq 1/2$. Thus the required condition for the original statement of \cite[Proposition 3.2]{Fujita2} is not fulfilled. Moreover, this also proves that 
$$
\alpha(\Sigma) \leq \frac{1}{2}.
$$
\end{remark}

\begin{remark}\label{rem:G-alpha}
Recall from \cite[Corollary 1.3]{Zhuang} that if a reductive algebraic group $G$ acting on a Fano variety $X$ of dimension $n$ and $\alpha_G(X) > n/(n+1)$, then $X$ is K-polystable. Lemmas \ref{lem:lct} and \ref{lem:ord_FD} show that 
$$
\alpha_{\Aut(\Sigma)}(\Sigma) = \frac{d+2}{2d}.
$$
However, if $d \geq 4$, then $(d+2)/(2d) \leq 2/3$. Thus one cannot prove K-polystability of $\Sigma$ in general only estimating $\alpha_{\Aut(\Sigma)}(\Sigma)$.
\end{remark}

\section{Anticanonical Polar Cylinders}\label{sec:cyl}
In this section, we prove Theorem \ref{thm:main2}. We keep using the notations introduced in the previous sections. It is enough to show that the secant variety $\Sigma$ of a rational normal curve of degree $d \geq 4$ admits a $(-K_\Sigma)$-polar cylinder. Recall that a \emph{$(-K_X)$-polar cylinder} on a $\Q$-Fano variety $X$ is a Zariski open subset $U:=X \setminus \operatorname{Supp}(D)$ for an effective divisor $D \sim_{\Q} -K_X$ such that $U \cong \C^1 \times Z$ for an affine variety $Z$.

\medskip

The universal family map $\sigma$ is explicitly given by
$$
\sigma \colon \nP^1 \times \nP^1 \longrightarrow \nP^2,~~[s,t] \times [u,v] \longmapsto \Big[su, \frac{1}{2}(tu+sv), tv \Big].
$$
Let $x,y,z$ be coordinates of $\nP^2$ such that 
$$
Q=\sigma(\Delta) = V(y^2-xz).
$$
Let $\ell := V(z)$ be the tangent line to $Q$ at $P=[1,0,0] \in Q$. Note that 
$$
\sigma^{-1}(\ell) = V(v) \cup V(t),
$$ 
where $V(v) = \nP^1 \times \{[1,0]\}$ and $V(t)=\{ [1,0]\} \times \nP^1$. Then
$$
\nP^1 \times \nP^1 \setminus \sigma^{-1}(\ell) = \C^1 \times \C^1 = \C^2
$$
with coordinates $s,u$. Note that $\nP^2 \setminus \ell = \C^2$ with coordinates $x,y$. The restriction map $\sigma|_{\C^2}$ is given by
$$
\sigma|_{\C^2} \colon \C^2 \longrightarrow \C^2,~~(s,u) \longmapsto \Big(su, \frac{1}{2}(s+u) \Big).
$$
Here we may think that $\C^2$ is the Hilbert scheme of two points on $\C^1 = \nP^1 \setminus \{[1,0]\}$. 

\medskip

Now, consider $\nP(E) \setminus \pi^{-1}(\ell) = \nP^1 \times \C^2$ with the projection
$$
\pi|_{\nP^1 \times \C^2} \colon \nP^1 \times \C^2 \longrightarrow \C^2,~~[w_0, w_1] \times (x,y) \longmapsto (x,y).
$$
Recall that $Z \setminus \pi^{-1}(\ell) = \nP^1 \times \nP^1 \setminus \sigma^{-1}(\ell) = \C^2$. The restriction of the inclusion $Z \subseteq \nP(E)$ is given by
$$
\iota \colon \C^2  \longrightarrow \nP^1 \times \C^2,~~(s,u) \longmapsto [s,1] \times \Big(su, \frac{1}{2}(s+u) \Big).
$$
Note that $s,u$ are zeros of the quadratic polynomial $X^2 - (s+u)X + su$ in $X$.
Thus 
$$
\iota(\C^2) = V(w_0^2 - 2y w_0 w_1 + x w_1^2).
$$

\medskip

We are ready to prove Theorem \ref{thm:main2}. We construct an effective $\Q$-divisor $D \sim_{\Q} -K_{\Sigma}$ such that $\Sigma \setminus \Supp(D)$ is a $(-K_{\Sigma})$-polar cylinder isomorphic to $\C^1 \times (\C^1 \times \C^*)$. 

\begin{proof}[Proof of Theorem \ref{thm:main2}]
Notice that 
$$
\beta(\pi^{-1}(\ell)) \sim_{\Q} -\frac{1}{2} K_{\Sigma}~~\text{ and }~~\pi^{-1}(\ell) + \frac{1}{d-2} Z = \beta^*(\beta(\pi^{-1}(\ell)) \sim_{\Q} \beta^* \Big(- \frac{1}{2}  K_{\Sigma} \Big).
$$
There is a nonzero section $H^0(\nP^1, \sO_{\nP^1}(d))$ vanishing at $[1,0] \in \nP^1$ of order $d$. As 
$$
H^0(\nP^1, \sO_{\nP^1}(d)) = H^0(\nP^2, E),
$$ 
this section gives a tautological divisor $H_0$ on $B=\nP(E)$. Then 
$$
\beta(H_0) \sim H \sim_{\Q} -\frac{d-2}{4} K_{\Sigma}.
$$

\medskip

Putting 
$$
D:=\frac{4}{d}(\beta(\pi^{-1}(\ell)) + \beta(H_0)) \sim_{\Q} -K_{\Sigma},
$$
we claim that
$$
\Sigma \setminus \Supp(D) = \C^1 \times (\C^1 \times \C^*).
$$
To this end, note that 
$$
\Sigma \setminus \Supp(D) = \Sigma \setminus (\beta(\pi^{-1}(\ell)) \cup \beta(H_0)) = B \setminus (\pi^{-1}(\ell) \cup Z \cup H_0).
$$
Recall that $B \setminus \pi^{-1}(\ell) = \nP^1 \times \C^2$. Observe that $H_0|_{\nP^1 \times \C^2} = V(w_1)$. Thus 
$$
B \setminus (\pi^{-1}(\ell) \cup H_0) = \C^1 \times \C^2 = \C^3
$$
with coordinates $w_0, x,y$. We have
$$
Z|_{\C^3} = \iota(\C^2)|_{\C^3} = V(w_0^2 - 2yw_0 + x).
$$
Now, consider the following isomorphism
$$
\gamma \colon \C^3 \longrightarrow \C^3,~~(w_0, x,y) \longmapsto (w_0, y, w_0^2-2yw_0 + x),
$$
under which $V(w_0^2 - 2yw_0 + x)$ correspondences to a coordinate plane of $\C^3$ isomorphic to $\C^2$.
Thus we see that
$$
\C^3 \setminus Z|_{\C^3} = \C^3 \setminus \gamma(Z|_{\C^3}) = \C^1 \times \C^1 \times \C^*.
$$
However, 
$$
\C^3 \setminus Z|_{\C^3} = \big( B \setminus (\pi^{-1}(\ell) \cup H_0) \big) \setminus Z = B \setminus (\pi^{-1}(\ell) \cup Z \cup H_0),
$$
so we complete the proof.
\end{proof}

$ $

\end{document}